\documentclass[12pt]{amsart}
\usepackage[utf8]{inputenc}
\usepackage{fullpage, amssymb}
\usepackage{hyperref}
\usepackage[capitalise]{cleveref}

\crefformat{equation}{(#2#1#3)}

\newtheorem{theorem}{Theorem}
\numberwithin{theorem}{section}

\newtheorem{proposition}[theorem]{Proposition}
\newtheorem{lemma}[theorem]{Lemma}
\newtheorem{corollary}[theorem]{Corollary}
\newtheorem{example}[theorem]{Example}

\newcommand{\Rr}{\mathbb{R}}
\newcommand{\x}{\textbf{x}}
\newcommand{\y}{\textbf{y}}

\newcommand{\dVol}{\operatorname{dVol}}
\newcommand{\Ric}{\operatorname{Ric}}
\newcommand{\Hess}{\operatorname{Hess}}

\title{Sharp $L^2$ estimates for the drift heat equation on shrinking Ricci solitons}
\author{Heather Macbeth}

\begin{document}

\begin{abstract}
  We prove an $L^2$ estimate for the drift heat equation on a complete gradient shrinking Ricci soliton.
This estimate has a time-dependent weight which is Gaussian in its spatial asymptotics.
When transferred and scaled to an estimate for the heat equation along the Ricci flow of the soliton,
this estimate is uniform up to the singular time.
\end{abstract}

\maketitle

\section{Introduction}\label{sec:intro}

\subsection{Setting}\label{sec:setup}

Let $(M,g,f)$ be a complete weighted manifold;
that is, a complete Riemannian manifold $(M,g)$
equipped with a ``weight'' function $f\in C^\infty(M, \mathbb{R})$.

Denote by $\Delta_f:=\Delta_g-\nabla^g f$  its \emph{drift Laplacian}.  
This partial differential operator is symmetric as an operator on $L^2(e^{-f}\dVol_g)$,
and it generates a semigroup $P_t:=e^{t\Delta_f}$ of operators on $L^2(e^{-f}\dVol_g)$.
For any $v_0 \in L^2(e^{-f}\dVol_g)$, the function $v(t, y):=(P_t(v_0))(y)$ is a solution to the \emph{drift heat equation}
\begin{equation}\label{drift-heat-equation}
  \frac{\partial v}{\partial t} = \Delta_fv
\end{equation}
with initial condition $v_0$.  See e.g.\ \cite[Theorem 3.1]{Gri06}.

A classical estimate
(which follows from a computation of the derivative of $\lVert P_tv\rVert_{L^2(e^{-f}\dVol_g)}^2$)
is that, for all $v\in L^2(e^{-f}\dVol_g)$ and all $t\geq 0$,
\begin{equation}\label{classical-bound}
\lVert P_tv\rVert_{L^2(e^{-f}\dVol_g)}\le \lVert v\rVert_{L^2(e^{-f}\dVol_g)}.
\end{equation}

The weighted manifold $(M,g,f)$ is a \emph{gradient shrinking Ricci soliton}, if there exists a positive real number $k$ such that $\Ric(g)+\Hess(f)=kg$.  
For simplicity we normalize throughout to the case of shrinking  solitons for which  $k=\tfrac{1}{2}$,
and also (see \cref{example-solutions-heat-eqn}) make a standard normalization for the soliton potential $f$
(eliminating the freedom to add a constant to $f$).  On such a soliton,
the potential function $f$ is of quadratic growth:
$\left\lvert f(y)-\frac{1}{4}r(y)^2\right\rvert \lesssim r(y)$,
where $r(y)$ denotes the distance from a fixed basepoint  \cite[Theorem 1.1]{CZ10}. 
Thus the measure $e^{-f}\dVol_g$ in the classical $L^2$ bound \cref{classical-bound} 
can be understood as the analogue of a Gaussian measure, relative to the Riemannian volume form $\dVol_g$.

\subsection{Pseudo-Gaussian $L^2$ bounds for the drift heat equation}\label{define-pseudo-gaussian}

The question studied in this article is to what extent $L^2$ bounds hold for the drift heat evolution operators $P_t$,
for other ``pseudo-Gaussian" measures on $(M, g, f)$.
By \emph{pseudo-Gaussian} we mean the measures
\[
e^{-f/\gamma}\dVol_g  
\]
on $M$, for positive real $\gamma$;
thus the classical result \cref{classical-bound} is a pseudo-Gaussian bound for $\gamma:=1$.
We emphasize that all these pseudo-Gaussian function spaces $L^2\left(e^{-f/\gamma}\dVol_g\right)$  contain functions with very fast
(e.g.\ superexponential) rates of spatial growth.

It turns out that a pseudo-Gaussian $L^2$ estimate can be obtained cheaply for any $\gamma<\tfrac{1+e^t}{2}$,
by combining the classical
Bakry-Emery hypercontractivity estimates \cite{BE85} with H\"older's inequality.
See \cref{bakry-emery-pseudo-gaussian}.

The main result of this article is a pseudo-Gaussian $L^2$ estimate in the case
$\gamma:=\tfrac{1+e^t}{2}$,
which \cref{bakry-emery-pseudo-gaussian} suggests to be the critical case.

\begin{theorem}\label{main-result}
  Let $v\in L^2(e^{-f}\dVol_g)$.
  Then for all $t\geq 0$,
  \[
    \lVert P_tv\rVert_{L^2\left(e^{-\frac{f}{\frac{1}{2}(e^t+1)}}\dVol_g\right)}
    \le
    e^{nt/4}\lVert v\rVert_{L^2(e^{-f}\dVol_g)}.
  \]
\end{theorem}
The proof is given in \cref{sec:critical}.
The key point is a derivative computation, generalizing that of \cref{classical-bound}.
We also give arguments showing two ways in which this result is optimal.

First, in \cref{example-solutions-heat-eqn},
we show that the constant $e^{nt/4}$ is optimal,
by giving explicit functions $v$ in the model case $M=\mathbb{R}^n$ for which 
\[
  e^{-nt/4}\lVert P_tv\rVert_{L^2\left(e^{-\frac{f}{\frac{1}{2}(e^t+1)}}\dVol_g\right)}  
\]
is arbitrarily close to $\lVert v\rVert_{L^2(e^{-f}\dVol_g)}$ as $t\to \infty$.

Second,
we carry out much of the argument for the proof of \cref{main-result} without specifying in advance which real $\gamma$ is to be taken for the pseudo-Gaussian bound.
It is apparent from the resulting ansatz that $\gamma:=\frac{1+e^t}{2}$ is optimal for this proof to go through.
See \cref{nonincreasing-sec}.

\subsection{Application to the heat equation along Ricci flow} \label{ricci-flow-setup}

We continue to work with a complete gradient shrinking Ricci soliton $(M, g, f)$, with standard normalizations.
Let $\psi_t$ be the one-parameter diffeomorphism associated to the vector field $\nabla f$.  
The one-parameter family of metrics
\[
g(\tau):=-\tau(\psi_{-\log(-\tau)})^*g,
\]
defined for $-\infty<\tau<0$, is a solution to the Ricci flow and has $g(- 1)=g$.

The importance of the drift Laplacian on a Ricci soliton is that it encapsulates statically the properties of the time-dependent heat equation along such a self-similar Ricci flow.
  Given a function $v(t,x)$ on $[0,\infty)\times M$, we introduce the function 
  \begin{equation} \label{ricci-heat-from-drift-heat}
  u(\tau, x):=v(-\log(-\tau), \psi_{-\log(-\tau)}(x))
  \end{equation}
  on $[-1,0)\times M$.
  Then $v(t,x)$ solves the drift heat equation \cref{drift-heat-equation}
  if and only if $u(\tau,x)$ solves the time-dependent heat equation along the Ricci flow $g(\tau)$,
  \begin{equation}
    \frac{\partial u}{\partial \tau}=\Delta_{g(\tau)}u.\label{heat-eqn-ricci-flow}
  \end{equation}

This relationship allows one to express any property of the drift  heat equation in terms of the heat equation along the Ricci flow, and vice versa.  
\cref{main-result} is interesting in this context because it scales correctly to give nontrivial behaviour along the Ricci flow.

\begin{corollary} \label{nonincreasing-along-flow}
  Let $u_{-1}\in L^2(M,e^{-f}\dVol_g)$.
  There exists a solution $u(\tau,x)$ on $[-1, 0)$ to the heat equation \cref{heat-eqn-ricci-flow} along the Ricci flow,
    with initial value $u|_{-1}\equiv u_{-1}$,
    and which satisfies, for all $\tau\in[0,1)$, that
\begin{align*}
\lVert u_{\tau}\rVert_{L^2\left(e^{-\frac{(-\tau)\psi_{-\log(-\tau)}{}^*f}{(1-\tau)/2}}\dVol_{g_\tau}\right)}
\le 
\lVert u_{-1}\rVert_{L^2\left(e^{-f}\dVol_g\right)}.
\end{align*}
\end{corollary}
The (short) proof is deferred to \cref{sec:transfer-bound}.

When the soliton $(M, g, f)$ is asymptotically conical with a sufficiently fast rate of convergence,
the rescaled metric and potential
\begin{align*}
  g(\tau)&=-\tau(\psi_{-\log(-\tau)})^*g,\\
  f(\tau)&:=(-\tau)\psi_{-\log(-\tau)}{}^*f
\end{align*}
converge locally uniformly on a dense open set 
to a Riemannian metric $G$ and function $F$ \cite[Section 4.3]{Sie13}, \cite[Section 3.1]{CDS19},
with the property that the metric space defined by $G$ is isometric to the asymptotic cone of $(M, g, f)$,
and the function $F$ is $\frac{1}{4}r^2$ in conical co-ordinates.

Thus the estimate \cref{nonincreasing-along-flow} is uniform up to the singular time $\tau=0$ of the Ricci flow,
where the limit measure is $e^{-2F}\dVol_G$.

We illustrate this in \cref{heat-operator-euclidean} by giving an explicit proof in the case of $\mathbb{R}^n$.
It is suggestive that, in that setting, estimates can be obtained on the whole interval $[-1,1)$;
that is, past the ``singular" time, $\tau=0$ (which of course, in the case of $\mathbb{R}^n$, does not exhibit a singularity).

\subsection{Acknowledgements}

I am grateful to Alix Deruelle for several pointers to the literature, and to Hans-Joachim Hein for many helpful discussions.
Part of this work was carried out during stays at the University of M\"unster;
I thank the mathematics department there for its hospitality.

\section{Pseudo-Gaussian bound from Bakry-Emery hypercontractivity estimates}\label{sec:bakry-emery}

In this  section we prove \cref{bakry-emery-pseudo-gaussian},
a cheap ``pseudo-Gaussian" (in the sense of \cref{define-pseudo-gaussian}) estimate
for the drift heat equation \cref{drift-heat-equation} which
stems from the hypercontractivity properties of the drift Laplacian semigroup.
This estimate serves as motivation for the stronger result of \cref{sec:critical}.

The result in ths section holds for a complete weighted manifold $(M,g,f)$ satisfying the one-sided Bakry-Emery bound 
\begin{equation}\label{bakry-emery-condition}
  \Ric(g)+\Hess(f)\geq \tfrac{1}{2}g.
\end{equation}
As in \cref{sec:setup},
we denote by $P_t:=e^{t\Delta_f}$ the heat evolution of the drift Laplacian $\Delta_f$.

\begin{theorem}[Bakry-Emery \cite{BE85}] \label{hypercontractive-two}
  For all $v\in L^2(e^{-f}\dVol_g)$, for all $t\in[0,\infty)$,   \[
    ||P_tv||_{L^{(1+e^t)/2}(e^{-f}\dVol_g)}\leq ||v||_{L^2(e^{-f}\dVol_g)}.
      \]
\end{theorem}

Suppose further that 
for all $\epsilon>0$, 
\begin{equation}\label{finite-vol-condition}
  \int_Me^{-\epsilon f}\dVol_g<\infty.
\end{equation}
For $q$ and $\gamma$ with $0<\gamma<q/2$, define
\[
C_{q,\gamma}:=
 \left\{\int_M\left(e^{-f\left[\tfrac{1}{\gamma}-\tfrac{2}{q}\right]}\right)^{\tfrac{q}{q-2}}\dVol_g\right\}^{\tfrac{q-2}{2q}};
\]
this is finite by \cref{finite-vol-condition}.

\begin{proposition}\label{bakry-emery-pseudo-gaussian}
  Let $v\in L^2(e^{-f}\dVol_g)$.  Then for all $t\geq 0$, for any $\gamma$ satisfying $\gamma<\tfrac{1+e^t}{2}$,
\begin{equation*}
  \lVert P_tv\rVert_{L^2(e^{-f/\gamma}\dVol_g)}
  \le C_{1+e^t,\gamma} \lVert v\rVert_{L^2(e^{-f}\dVol_g)}.
\end{equation*}
\end{proposition}

The constants $C_{1+e^t,\gamma}$ are not controlled as $\gamma\to \tfrac{1+e^t}{2}$; see discussion after the proof.

\begin{proof}
  We will take $C := C_{1+e^t,\gamma}$.
Let $q:=1+e^{t}$.
Then, for all $0<\gamma<q/2=\tfrac{1+e^{t}}{2}$, 
by H\"older's inequality,
\begin{align*}
  \int_M(P_tv)^2e^{-f/\gamma}\dVol_g
  &=\int_M
  \left((P_tv)^qe^{-f}\right)^{\frac{2}{q}}
  \left(e^{-f\left[\tfrac{1}{\gamma}-\tfrac{2}{q}\right]}\right)\\
  &\leq 
  \left\{\int_M(P_tv)^qe^{-f}\dVol_g\right\}^{\frac{2}{q}}
  \left\{\int_M\left(e^{-f\left[\tfrac{1}{\gamma}-\tfrac{2}{q}\right]}\right)^{\tfrac{q}{q-2}}\dVol_g\right\}^{\frac{q-2}{q}}\\
  &= ||P_tv||_{1+e^{t}}^2\cdot C_{1+e^t,\gamma}^2\\
  &\le ||v||_{2}^2\cdot C_{1+e^t,\gamma}^2\\
  &= C_{1+e^t,\gamma}^2 \int_Mv^2e^{-f}\dVol_g.
\end{align*}
the second-last step by \cref{hypercontractive-two}.

\end{proof}

\begin{example}
  On $\mathbb{R}^n$, with $f(y)=\tfrac{1}{4}\lVert y\rVert^2$,
\begin{align*}
  C_{1+e^t,\gamma}^2
  &=
 \left\{\int\left(e^{-f\left[\tfrac{1}{\gamma}-\tfrac{2}{1+e^t}\right]}\right)^{\tfrac{1+e^t}{(1+e^t)-2}}\dVol_g\right\}^{\tfrac{(1+e^t)-2}{1+e^t}}\\
 &=
 \left\{\left(\tfrac{1}{4}\left[\tfrac{1}{\gamma}-\tfrac{2}{1+e^t}\right]\tfrac{1+e^t}{e^t-1}\right)^{-1}\pi\right\}^{\tfrac{e^t-1}{1+e^t}\tfrac{n}{2}}.
\end{align*}

Consider taking $\gamma(t)=\tfrac{1+e^t-\epsilon(e^t-1)}{2}$, for some small $\epsilon$. Then
\begin{align*}
  \frac{1}{\gamma(t)}-\frac{2}{1+e^t}
  &=2\left[\frac{1}{1+e^t-\epsilon(e^t-1)}-\frac{1}{1+e^t}\right]\\
  &=\frac{2\epsilon(e^t-1)}{(1+e^t)(1+e^t-\epsilon(e^t-1))}.
\end{align*}
So
\begin{align*}
  C_{1+e^t,\gamma(t)}^{4/n}
 &=
 \left\{\left(\tfrac{1}{4}\left[\tfrac{1}{\gamma}-\tfrac{2}{1+e^t}\right]\tfrac{1+e^t}{e^t-1}\right)^{-1}\pi\right\}^{\tfrac{e^t-1}{1+e^t}}\\
  &=
  \left\{\tfrac{2\pi(1+e^t-\epsilon(e^t-1))}{\epsilon}\right\}^{\tfrac{e^t-1}{1+e^t}}\\
  e^{-t}C_{1+e^t,\gamma(t)}^{4/n}
  &=
  \left\{\tfrac{2\pi(e^{-t}+1-\epsilon(1-e^{-t}))}{\epsilon}\right\}^{\tfrac{1-e^{-t}}{e^{-t}+1}}.
\end{align*}
This tends to $\frac{2\pi(1-\epsilon)}{\epsilon}$ as $t\to \infty$.

\end{example}

\section{Examples}\label{example-solutions-heat-eqn}

In this section we note some explicit solutions $v(t,y)$ to the drift heat equation \cref{drift-heat-equation},
either on a general shrinking soliton or on $\mathbb{R}^n$, defined for $t\in[0,\infty)$,
together with the associated solution $u(\tau,x)$ (see \cref{ricci-flow-setup}) to the heat equation  \cref{heat-eqn-ricci-flow} along the Ricci flow.
As well as providing some useful intuition on unbounded solutions to the heat equation,
this analysis demonstrates that \cref{main-result} is sharp.

As noted in \cref{sec:setup}, we work with shrinking solitons $(M, g, f)$
on which the soliton constant is normalized to $\frac{1}{2}$
and the potential function is normalized (by the addition of a constant) so that all three of the standard identities for the soliton potential \cite[(4.9), (4.13)]{CLN06} are valid:
\begin{equation}
  \left.
\begin{aligned}
  R_g+\Delta f &= \frac{n}{2}\\
R_g+\lvert df\rvert^2&=f\\
\Delta f - \lvert df\rvert^2&=\frac{n}{2}-f.
\end{aligned}
\right\rbrace\label{soliton-identities}
\end{equation}

  In the case of $\mathbb{R}^n$, with $f(y)=\tfrac{1}{4}\lVert y\rVert^2$,
  the drift heat equation is
  $\frac{dv}{dt}=\Delta v - \tfrac{1}{2}y\cdot \nabla v$
  and the heat equation along the Ricci flow is just the standard heat equation on $\mathbb{R}^n$.
  We also exploit that the translation between the  $v(t,y)$ and  $u(\tau,x)$ columns is more explicit than in \cref{ricci-flow-setup}:
\begin{align*}
  u(\tau,x)&=v(-\log(-\tau),(-\tau)^{-1/2}x)\\
  v(t,y)&=u(-e^{-t},e^{-t/2}y).
\end{align*}

Here, then, are some explicit solutions to the two equations.
\vspace{5mm}

\newcounter{rowcntr}
\renewcommand{\therowcntr}{\arabic{rowcntr}}

\begin{tabular}{l|l|l|l|l}
  & setting & $v(t,y)$ & $u(\tau,x)$ & spatial asymptotics \\
  \hline
  \hline
  \refstepcounter{rowcntr}\therowcntr \label{zero-eigenfunction} & general  & $1$ & $1$ & constant \\
  \hline
  \refstepcounter{rowcntr}\therowcntr \label{one-eigenfunction} & $\mathbb{R}^n$ & $e^{-t/2}c\cdot y$ & $c\cdot x$ & linear  \\
  \hline
  \refstepcounter{rowcntr}\therowcntr \label{two-eigenfunction-euclidean} & $\mathbb{R}^n$ 
  & $e^{-t}\left(\frac{1}{4}\lVert y\rVert^2-\frac{n}{2}\right)$ 
  & $\frac{1}{4}\lVert x\rVert^2+\frac{n\tau}{2}$ & quadratic \\
  \hline
  \refstepcounter{rowcntr}\therowcntr \label{two-eigenfunction} & general  
  & $e^{-t}\left(f-\frac{n}{2}\right)$ 
  & $-\tau f \circ \psi_{-\log(-\tau)}+\frac{n\tau}{2}$ & quadratic\\
  \hline
  \refstepcounter{rowcntr}\therowcntr \label{eigenfunction} 
  & general  & $e^{-\lambda t}e_\lambda $ 
  & $(-\tau)^\lambda e_\lambda \circ \psi_{-\log(-\tau)}$ & polynomial\footnote{conjectured} \\
  \hline
  \refstepcounter{rowcntr}\therowcntr \label{reverse-heat-kernel-function} 
  & $\mathbb{R}^n$ 
  & $\left(c+e^{-t}\right)^{-n/2}e^{\frac{\lVert y\rVert^2}{4(ce^t+1)}}$ 
  & $\left(c-\tau\right)^{-n/2}e^{\frac{\lVert x\rVert^2}{4(c-\tau)}}$
  & superexponential \\
  \hline
  \refstepcounter{rowcntr}\therowcntr \label{heat-kernel-function} & $\mathbb{R}^n$ 
  & $\left(c-e^{-t}\right)^{-n/2}e^{-\frac{\lVert y\rVert^2}{4(ce^{t}-1)}}$
  & $\left(c+\tau\right)^{-n/2}e^{-\frac{\lVert x\rVert^2}{4(c+\tau)}}$
  & Gaussian
\end{tabular}

\vspace{5mm}

In (\ref{eigenfunction}), $e_\lambda$ represents a $\lambda$-eigenfunction for the drift Laplacian $\Delta_f$. In (\ref{reverse-heat-kernel-function}), the constant $c$ should be nonnegative; in (\ref{heat-kernel-function}), the constant $c$ should be at least 1.

Note that 
\begin{itemize}
  \item (\ref{two-eigenfunction}) generalizes (\ref{two-eigenfunction-euclidean}); it follows from one of the soliton identities (\ref{soliton-identities});
  \item (\ref{eigenfunction}) generalizes (\ref{zero-eigenfunction}) (with $\lambda =0$), (\ref{one-eigenfunction}) (with $\lambda=\frac{1}{2}$), and (\ref{two-eigenfunction-euclidean})-(\ref{two-eigenfunction}) (with $\lambda=1$);
  \item all these functions $v(t,y)$ tend pointwise to constants as $t\to\infty$.
\end{itemize}

Now we look at the behavior of a few of these quantities of interest over time, for the case of $\mathbb{R}^n$ where they can be computed explicitly.
Note that Example \ref{reverse-heat-kernel-function} makes sense on the desired time span when the constant $c$ is taken to have $c>1$.

\vspace{5mm}

\begin{tabular}{l|l|l|l}
  & $v(t,y)$ & $\int v^2e^{-\frac{\lVert y\rVert^2}{4\gamma(t)}}dy$ & $e^{-nt/2}\int v^2e^{-\frac{\lVert y\rVert^2}{4(e^t+1)/2}}dy$ \\
  \hline
  \hline
  \ref{zero-eigenfunction} & $1$ 
  & $\left[4\pi\gamma(t)\right]^{n/2}$
  & $\left[2\pi(1+e^{-t})\right]^{n/2}$ \\
  \hline
  \ref{one-eigenfunction} & $e^{-t/2}c\cdot y$
  & $2^{n+1}\pi^{n/2}\lVert c\rVert^2e^{-t}\gamma(t)^{n/2+1}$
  & $(2\pi)^{n/2}\lVert c\rVert^2\left(1+e^{-t}\right)^{n/2+1}$ \\
  \hline
  \ref{reverse-heat-kernel-function} 
  & $\left(c+e^{-t}\right)^{-n/2}e^{\frac{\lVert y\rVert^2}{4(ce^t+1)}}$ 
  & $\left(\frac{2\pi e^t}{\left[\frac{ce^t+1}{2\gamma(t)}-1\right]\left[c+e^{-t}\right]}\right)^{n/2}$
  & $\left(\frac{2\pi}{c-1}\right)^{n/2}\left(\frac{1+e^{-t}}{c+e^{-t}}\right)^{n/2}$
\end{tabular}

\vspace{5mm}

We can consider these examples in relation to the bound on the operator norm of the evolution map from $t=0$ to large $t$ given by \cref{main-result}, namely that it is at most 1.
Each individual quantity
\[
L(v):=\sqrt{\frac{\lim_{t\to\infty}e^{-nt/2}\int v^2e^{-\frac{\lVert y\rVert^2}{4(e^t+1)/2}}dy}{\int v^2e^{-\frac{\lVert y\rVert^2}{4}}dy}}
\]
gives a lower bound for the asymptotic behaviour of these operator norms relative to $e^{nt/4}$ (the constant obtained in \cref{main-result}).

\vspace{5mm}

\begin{tabular}{l|l|l}
  & $v(t,y)$ & $L(v)$ \\
  \hline
  \hline
  \ref{zero-eigenfunction}
  & 1 & $\frac{1}{2^{n/4}}$ \\
  \hline
  \ref{one-eigenfunction} & $e^{-t/2}c\cdot y$ & $\frac{1}{2^{(n+2)/4}}$ \\
  \hline
  \ref{reverse-heat-kernel-function} 
  & $\left(c+e^{-t}\right)^{-n/2}e^{\frac{\lVert y\rVert^2}{4(ce^t+1)}}$ 
  & $\left(\frac{c+1}{2c}\right)^{n/4}$
  \\
\end{tabular}

\vspace{5mm}

The bound from Example \ref{reverse-heat-kernel-function} is most useful:
choosing $c$ arbitrarily close to 1 (subject to the constraint $c>1$) shows that the $L(v)$ can be made arbitrarily close to 1,
so \cref{main-result} is sharp on $\mathbb{R}^n$, at least as $t\to\infty$.
  
\section{Pseudo-Gaussian bound in the critical case}\label{sec:critical}

In this section we prove \cref{main-result}.
Almost all the work is done in the setting of initial data which is uniformly bounded to second order,
culminating in \cref{optimal-pseudo-gaussian},
the specialization of \cref{main-result} to this setting.
In the final part of the argument, we extend by approximation to general initial data.

The steps of the argument are:
\begin{enumerate}
\item a proof that such bounds on initial data extend to spatial bounds for all time
(\cref{subsec:spatial-bounds});\\
\item a calculation of the derivative of
\[
  \mu(t)^{n/2}\lVert P_tv\rVert_{L^2\left(e^{-f/\gamma(t)}\dVol_g\right)}^2
\]
for potentially time-dependent $\mu$ and $\gamma$
(which uses the bounds from \cref{subsec:spatial-bounds} to ensure the validity of an application of Stokes'
 theorem);
 \item the specialization to shrinking Ricci solitons and the main result (\cref{nonincreasing-sec}).
\end{enumerate}

\subsection{Estimates given bounded initial data} \label{subsec:spatial-bounds}

The main result of this subsection is \cref{uniform-bounds},
spatial bounds for all time.
This is probably well-known but we include an argument for completeness.

We recall the Karp-Li-type maximum principle for heat equations,
see \cite[Theorem 1.2]{NT04}, \cite[Theorems 7.39, 7.43]{CLN06} for general expositions.
We adapt a version stated by Li \cite[Theorem 5.9]{Li15} (taking $\alpha(r)\equiv 1$ in the statement there).

Let $(M,g, f)$ be a complete weighted manifold satisfying 
\begin{equation}
  \int_Me^{-f}\dVol_g<\infty.\label{true-finite-volume-condition}
\end{equation}

\begin{proposition}\label{maximum-principle}
  Let $v(y,t)$ be a smooth function on $M\times [0, T]$ such that
  \[
  \left(\Delta_f-\frac{d}{d t}\right)v\ge 0,
  \]
  (that is, $v$ is a subsolution to the drift heat equation), and at the initial time $v(\cdot,0)\le C$.
  If moreover $v$ satisfies the size condition
  \begin{align}
    \int_0^{T}\int_M v(y,s)^2 e^{-f(x)}\dVol_{g} ds &<\infty,\label{time-l2-bound}
  \end{align}
  then $v(\cdot ,t)\le C$  for all $t\in [0,T]$.
\end{proposition}

(Specifically, we apply the version of \cite{Li15} to the function $v-C$,
noting that $C$ is a solution to the drift heat equation
and that by \cref{true-finite-volume-condition} the condition \cref{time-l2-bound} on $v$ holds also for $v-C$.)

Next, let $(M, g, f)$ be a complete weighted manifold satisfying the lower Bakry-Emery curvature bound
\begin{equation}
  \Ric(g)+\Hess f\geq -kg,\label{nonnegative-bakry-emery}
\end{equation}
for some $k\geq 0$. (Temporarily, we do not assume \cref{true-finite-volume-condition}.)

\begin{lemma}\label{gradient-subsolution}
  Let $v(y,t)$ be a solution to the drift heat equation \cref{drift-heat-equation}.
  Then $e^{-kt}\lvert dv\rvert^2$ is a subsolution.
\end{lemma}

\begin{proof}

We have the Bochner identity
\[
\frac{1}{2}\Delta_f\lvert du\rvert^2
=\left\lvert \Hess u\rvert^2+\langle du, d(\Delta_fu)\right\rangle + \Ric_f(du, du).
\]
Also 
\[
\frac{1}{2}\frac{d}{dt}\lvert du\rvert^2
=\left\langle du, d\left(\frac{du}{dt}\right)\right\rangle.  
\]
So
\begin{align*}
\frac{1}{2}\left[\Delta_f-\frac{d}{dt}\right]\left(e^{-kt}\lvert du\rvert^2\right)
&=e^{-kt}\left\{\lvert \Hess u\rvert^2
+\left\langle du, d\left(\left[\Delta_f-\frac{d}{dt}\right]u\right)\right\rangle 
+ \Ric_f(du, du)
+k\lvert du\rvert^2\right\}\\
&=e^{-kt}\left\{\lvert \Hess u\rvert^2
+ \left(\Ric_f+kg\right)(du, du)\right\}\\
&\geq 0,
\end{align*}
the last line by \cref{nonnegative-bakry-emery}.

\end{proof}

Finally, let $(M, g, f)$ be a complete weighted manifold satisfying \cref{true-finite-volume-condition} and \cref{nonnegative-bakry-emery}.
Let $v_0$ be a $C^2$ function on $M$ satisfying, for some constants $C$, the uniform bounds
\begin{equation}
  \left.
\begin{aligned}
|v_0|&\le C\\
\left\lVert dv_0\right\rVert_g &\le C\\
\left\lvert \Delta_f v_0\right\rvert &\le C.
\end{aligned}
\right\rbrace\label{uniform-initial-bounds}
\end{equation}
This function belongs to $L^2(e^{-f}\dVol_g)$ (by \cref{true-finite-volume-condition} and the sup-bound).
So let $v(t,x)$ be the unique smooth solution in $L^2(e^{-f}\dVol_g)$ to
\[
\begin{cases}\label{IVP}
    \frac{dv}{dt}&=\Delta_fv\\
    v|_0&\equiv v_0,
  \end{cases}
\]
which is guaranteed by the theory sketched in \cref{sec:setup}.

\begin{proposition}\label{uniform-bounds}
  For the same constants, for all time,
\begin{align*}
    |v|&\le C\\
    \left\lVert dv\right\rVert_g &\le e^{kt}C\\
    \left\lvert \Delta_f v\right\rvert &\le C.
\end{align*}
  \end{proposition}

\begin{proof}
For each bound we exploit the maximum principle, \cref{maximum-principle}.
We obtain the first bound by applying it to $v$ and $-v$,
the second bound by applying it to $e^{-kt}\left\lVert dv\right\rVert_g^2$,
and the third bound by applying it to $\Delta_fv$ and $-\Delta_fv$.

This is valid since  $e^{-kt}\left\lVert dv\right\rVert_g^2$ is a subsolution (by \cref{gradient-subsolution}),
and $\Delta_fv$ is a solution: indeed,
\[
\frac{d}{dt}\left[\Delta_fv\right]
  =\Delta_f\left[\frac{dv}{dt}\right]
  =\Delta_f\left[\Delta_fv\right].
\]

\end{proof}

\subsection{Derivative computation} \label{subsec:derivative-computation}

\begin{lemma}\label{find-divergence-2}
\begin{align*}
&2v
\Delta_f v 
e^{-\frac{f}{\gamma}}
+2\left\lVert 
dv-\frac{-1+\alpha+\gamma^{-1}}{2}vdf
\right\rVert ^ 2e^{-\frac{f}{\gamma}}
-
\left\{
  \left[
    \frac{1}{2}(1-\alpha)^2
    +\gamma^{-2}\left(\frac{1}{2}-\gamma\right)
  \right]
\left\lVert df\right\rVert ^ 2
+\alpha \Delta_g f
\right\}
v^2e^{-\frac{f}{\gamma}}\\
&=\operatorname{div}_g\left[
\left(2vdv-\alpha v^2df\right)
e^{-\frac{f}{\gamma}}
\right].
\end{align*}
\end{lemma}

\begin{proof}
\begin{align*}
  \left\lVert 
dv-\frac{-1+\alpha+\gamma^{-1}}{2}vdf
\right\rVert ^ 2
&=
\left\lVert dv\right\rVert ^ 2
-(-1+\alpha+\gamma^{-1})v
\left\langle
  df,dv\right\rangle
  + \left(\frac{-1+\alpha+\gamma^{-1}}{2}\right)^2v^2
   \left\lVert 
    df
    \right\rVert ^ 2,
\end{align*}
so the coefficients, respectively, of
\[
v\Delta_g ve^{-\frac{f}{\gamma}},\quad
\lVert dv\rVert^2 e^{-\frac{f}{\gamma}},\quad
(\Delta_gf) v^2 e^{-\frac{f}{\gamma}},\quad
v\langle df, dv\rangle e^{-\frac{f}{\gamma}},\quad
v^2\lVert df\rVert^2e^{-\frac{f}{\gamma}}
\]
in the left-hand side are 2, 2, $-\alpha$,
\begin{align*}
-2-2(-1+\alpha+\gamma^{-1})
&=
-2\alpha-2\gamma^{-1},
\end{align*}
and 
\begin{align*}
&2\left(\frac{-1+\alpha+\gamma^{-1}}{2}\right)^2
-   \left[
  \frac{1}{2}(1-\alpha)^2
  +\gamma^{-2}\left(\frac{1}{2}-\gamma\right)
\right]\\
&=
\frac{1}{2}(-1+\alpha)^2
+(-1+\alpha)\gamma^{-1}
+\frac{1}{2}\gamma^{-2}
-   \left[
  \frac{1}{2}(1-\alpha)^2
  +\frac{1}{2}\gamma^{-2}
  -\gamma^{-1}
\right]
\\
&=\alpha\gamma^{-1}.
\end{align*}

On the right-hand side, we have,
\begin{align*}
\operatorname{div}_g\left(
vdve^{-\frac{f}{\gamma}}
\right)
&=
\lVert dv\rVert^2 e^{-\frac{f}{\gamma}}
+v\Delta_g ve^{-\frac{f}{\gamma}}
-\gamma^{-1}
v\langle df, dv\rangle e^{-\frac{f}{\gamma}},
\\
\operatorname{div}_g\left(
v^2df
e^{-\frac{f}{\gamma}}
\right)&=
2v\langle df, dv\rangle e^{-\frac{f}{\gamma}}
+(\Delta_gf) v^2 e^{-\frac{f}{\gamma}}
-\gamma^{-1}
v^2\lVert df\rVert^2e^{-\frac{f}{\gamma}},
\end{align*}
so the coefficients, respectively, of
\[
v\Delta_g ve^{-\frac{f}{\gamma}},\quad
\lVert dv\rVert^2 e^{-\frac{f}{\gamma}},\quad
(\Delta_gf) v^2 e^{-\frac{f}{\gamma}},\quad
v\langle df, dv\rangle e^{-\frac{f}{\gamma}},\quad
v^2\lVert df\rVert^2e^{-\frac{f}{\gamma}}
\]
in the right-hand side $\operatorname{div}_g\left[
  \left(2vdv-\alpha v^2df\right)
  e^{-\frac{f}{\gamma}}\right]$
are 2, 2, $-\alpha$,
$-2\gamma^{-1}-2\alpha$, and $\alpha\gamma^{-1}$.

The coefficients agree, so the identity holds.

\end{proof}

Let $(M, g, f)$ be a complete weighted manifold satisfying 
that for all real $\gamma>0$,
\begin{equation}\label{integral-bounds-f-derivs}
  \left.
  \begin{aligned}
    \int_M
    e^{-\frac{f}{\gamma}}\dVol_g&<\infty\\
    \int_M
  \left\lVert df\right\rVert_g ^ 2
  e^{-\frac{f}{\gamma}}\dVol_g&<\infty\\
  \int_M
  \left\lvert\Delta_g f\right\rvert
  e^{-\frac{f}{\gamma}}\dVol_g&<\infty.
\end{aligned}\right\rbrace
\end{equation}

\begin{proposition}\label{use-divergence}
  Let $v\in C^2(M)$ be a function satisfying the bounds \cref{uniform-initial-bounds}.
 Then for any real $\alpha$ and any positive real $\gamma$,
  \begin{align*}
  &\int_M2v
  \Delta_f v 
  e^{-\frac{f}{\gamma}}\dVol_g\notag\\
  &= % ********************************
  \int_M
  -2\left\lVert 
  dv-\frac{-1+\alpha+\gamma^{-1}}{2}vdf
  \right\rVert ^ 2e^{-\frac{f}{\gamma}}
  +
  \left\{
    \left[
      \frac{1}{2}(1-\alpha)^2
      +\gamma^{-2}\left(\frac{1}{2}-\gamma\right)
    \right]
  \left\lVert df\right\rVert ^ 2
  +\alpha \Delta_g f
  \right\}
  v^2e^{-\frac{f}{\gamma}}
  \dVol_g.
  \end{align*}
\end{proposition}

\begin{proof}
We first note that both integrands are integrable.
Indeed, note that
\begin{align*}
  \left\lVert 
    dv-\frac{-1+\alpha+\gamma^{-1}}{2}vdf
  \right\rVert ^ 2
  &\le C\left[
  \left\lVert 
    dv \right\rVert ^ 2
  +v^2\left\lVert df
  \right\rVert ^ 2\right];
\end{align*}
we have by \cref{uniform-initial-bounds} that $v$, $\left\lVert dv\right\rVert ^ 2$ and $\Delta_fv$ are bounded,
so by \cref{integral-bounds-f-derivs} both integrands are integrable.

Therefore their difference, 
\begin{align*}
  &\int_M
  \left\{
    2v
    \Delta_f v 
    e^{-\frac{f}{\gamma}}
    +2\left\lVert 
    dv-\frac{-1+\alpha+\gamma^{-1}}{2}vdf
    \right\rVert ^ 2e^{-\frac{f}{\gamma}}\right.\\
   &\qquad \left.-
    \left(
      \left[
        \frac{1}{2}(1-\alpha)^2
        +\gamma^{-2}\left(\frac{1}{2}-\gamma\right)
      \right]
    \left\lVert df\right\rVert ^ 2
    +\alpha \Delta_g f
    \right)
    v^2e^{-\frac{f}{\gamma}}
    \right\}\dVol_g,
  \end{align*}
is integrable and it suffices to prove that this integral is zero.
By \cref{find-divergence-2} the integrand is equal to
  $\operatorname{div}_g\left[
    \left(2vdv-\alpha v^2df\right)
    e^{-\frac{f}{\gamma}}
    \right]$,
so it suffices to prove that 
\begin{align*}
\int_M
\operatorname{div}_g\left[
    \left(2vdv-\alpha v^2df\right)
    e^{-\frac{f}{\gamma}}
    \right]
    \dVol_g&=0.
\end{align*}

Next, observe that 
\begin{align*}
  \left\lVert\left(2vdv-\alpha v^2df\right)
  e^{-\frac{f}{\gamma}}\right\rVert
  &\le \left[2|v|\left\lVert dv\right\rVert+|\alpha|v^2\left\lVert df\right\rVert\right] e^{-\frac{f}{\gamma}},
\end{align*}
so by \cref{uniform-bounds} and \cref{integral-bounds-f-derivs} the quantity
$\left(2vdv-\alpha v^2df\right)
  e^{-\frac{f}{\gamma}}$
is integrable.
By \cite[section 1]{Yau76} this implies the existence of an
exhaustion $(U_i)_{i\in\mathbb{N}}$ of $M$ such that
\begin{align*}
\lim_{i\to\infty}\int_{U_i}
\operatorname{div}_g\left[
  \left(2vdv-\alpha v^2df\right)
  e^{-\frac{f}{\gamma}}
  \right]
\dVol_g&=0.
\end{align*}
Since $\operatorname{div}_g\left[
    \left(2vdv-\alpha v^2df\right)
    e^{-\frac{f}{\gamma}}
    \right]$
is also integrable, by the dominated convergence theorem this implies that 
\begin{align*}
  \int_{M}
  \operatorname{div}_g\left[
    \left(2vdv-\alpha v^2df\right)
    e^{-\frac{f}{\gamma}}
    \right]
  \dVol_g&=0,
  \end{align*}
as required.

\end{proof}

Now, suppose that the weighted manifold $(M, g, f)$ also satisfies \cref{nonnegative-bakry-emery} (Bakry-Emery-Ricci curvature bounded below).

\begin{proposition}\label{compute-derivative}
  Let $v_0\in C^2(M)$ be a function satisfying the bounds \cref{uniform-initial-bounds},
  and let $v(t,y)$ be the $L^2(e^{-f}\dVol_g)$ solution to the drift heat equation with initial condition $v_0$.
  Then for  positive $\gamma(t)$ and $\mu(t)$, and for any real $\alpha$,
  \[
    \mu(t)^{-n/2}
    \frac{d}{dt}
    \left(
      \mu(t)^{n/2}\int v(t,\cdot)^2e^{-\frac{f}{\gamma(t)}}\dVol_g
    \right)
  \]
  is less than or equal to
  \begin{align*}
    &\int_M
    \left\{
      \left[
        \frac{1}{2}(1-\alpha)^2
        +\gamma^{-2}\left(\frac{1}{2}-\gamma\right)
      \right]
    \left\lVert df\right\rVert ^ 2
    +\alpha \Delta_g f
    +\frac{\dot\gamma}{\gamma^2}f
    +\frac{\dot\mu}{\mu}\frac{n}{2}
  \right\}
    v^2e^{-\frac{f}{\gamma}}
    \dVol_g.
  \end{align*}
\end{proposition}

\begin{proof}
\begin{align*}
&\frac{d}{dt}
\left(
  \mu(t)^{n/2}\int v(t,\cdot)^2e^{-\frac{f}{\gamma(t)}}\dVol_g
\right)\\
&=
\mu(t)^{n/2}\frac{d}{dt}
\left(
\int v^2e^{-\frac{f}{\gamma(t)}}\dVol_g
\right)
+\frac{n}{2}\dot\mu(t)\mu(t)^{n/2-1}
\left(\int v^2e^{-\frac{f}{\gamma(t)}}\dVol_g
\right)\\
&= % *****************************
\mu(t)^{n/2}
\int
  \left\{
  2v\frac{dv}{dt}
  +\frac{\dot\gamma(t)}{\gamma(t)^2}fv^2
  +\frac{\dot\mu(t)}{\mu(t)}\frac{n}{2}v^2
  \right\}
  e^{-\frac{f}{\gamma(t)}}\dVol_g\\
&= % *****************************
\mu(t)^{n/2}
\int
  \left\{
  2v\Delta_fv
  +\frac{\dot\gamma(t)}{\gamma(t)^2}fv^2
  +\frac{\dot\mu(t)}{\mu(t)}\frac{n}{2}v^2
  \right\}
  e^{-\frac{f}{\gamma(t)}}\dVol_g,
\end{align*}
the last line since $\frac{dv}{dt}=\Delta_fv$.

By \cref{uniform-bounds}, \cref{use-divergence} applies to $v(t,\cdot)$.
So we have,
\begin{align*}
  \int_M2v
  \Delta_fv
  e^{-\frac{f}{\gamma}}\dVol_g
  &\le 
  \int_M
  \left\{
    \left[
      \frac{1}{2}(1-\alpha)^2
      +\gamma^{-2}\left(\frac{1}{2}-\gamma\right)
    \right]
  \left\lVert df\right\rVert ^ 2
  +\alpha \Delta_g f
  \right\}
  v^2e^{-\frac{f}{\gamma}}
  \dVol_g.
  \end{align*}
The result follows.
\end{proof}

\subsection{Proof of \cref{main-result}} \label{nonincreasing-sec}

Suppose that $(M,g,f)$ is a gradient shrinking Ricci soliton.
This implies \cref{integral-bounds-f-derivs} \cite{CZ10}.
As usual, for simplicity we take the soliton constant to be $\tfrac{1}{2}$.

We will use the following ansatz to analyze the possible pseudo-Gaussian bounds for the drift heat equation.

\begin{proposition}\label{complicated-derivative}
  Let $v\in C^2(M)$ be a function satisfying the bounds \cref{uniform-initial-bounds}.
  Then for  positive $\gamma(t)$ and $\mu(t)$, and any real $\alpha$, for all $t\geq 0$,
  \[
    \mu(t)^{-n/2}
    \frac{d}{dt}
    \left(
      \mu(t)^{n/2}\int (P_tv)^2e^{-\frac{f}{\gamma(t)}}\dVol_g
    \right)
  \]
  is less than or equal to
  \begin{align*}
    \int_M
    \left\{
      \left[
    \frac{1}{2}(1-\alpha)^2
    +\gamma^{-2}\left(\frac{1}{2}-\gamma+\dot\gamma\right)
    \right]f
    +\left(\alpha+\frac{\dot\mu}{\mu}\right)\frac{n}{2}
  \right\}
    (P_tv)^2e^{-\frac{f}{\gamma}}
    \dVol_g.
    \end{align*}
\end{proposition}

\begin{proof}
By the soliton identities \cref{soliton-identities},
\begin{align*}
\lVert df\rVert^2&=f-R_g\\
\Delta_g f &= \frac{n}{2}-R_g.
\end{align*}
The result follows by substituting these into \cref{compute-derivative},
and noting that the resulting coefficient for $R_g$ is
\begin{align*}
  -\left[
    \frac{1}{2}(1-\alpha)^2
    +\gamma^{-2}\left(\frac{1}{2}-\gamma\right)
  \right]
  + \alpha
&=-\frac{1}{2}\left[
  \alpha^2
  +(\gamma^{-1}-1)^2\right]\\
  &\le 0.
\end{align*}

Since $R_g$ is nonnegative \cite[Corollary 2.5]{Chen09}, the term  
\[
  -\left[
    \frac{1}{2}(1-\alpha)^2
    +\gamma^{-2}\left(\frac{1}{2}-\gamma\right)
  \right]
  R_g
\]
is nonpositive.
\end{proof}

As an illustration that this ansatz is sufficiently flexible,
we first retrieve the classical bound \cref{classical-bound}, the standard (time-invariant) pseudo-Gaussian bound for the drift heat equation.

\begin{example}\label{time-invariant-pseudo-gaussian}
  Let $v\in C^2(M)$ be a function satisfying the bounds \cref{uniform-initial-bounds}.
  Then for all $t\geq 0$,
  \[
    \lVert P_tv\rVert_{L^2(e^{-f}\dVol_g)}
    \le
    \lVert v\rVert_{L^2(e^{-f}\dVol_g)}.
  \]
\end{example}

\begin{proof}
  We will show that
  \[
    \frac{d}{dt}
      \int (P_tv)^2e^{-f}\dVol_g
  \]
  is nonpositive.

  By \cref{complicated-derivative}, it suffices to show that 
  \begin{align}
    \int_M
    \left\{
      \left[
    \frac{1}{2}(1-\alpha)^2
    +\gamma^{-2}\left(\frac{1}{2}-\gamma+\dot\gamma\right)
    \right]f
    +\left(\alpha+\frac{\dot\mu}{\mu}\right)\frac{n}{2}
  \right\}
    (P_tv)^2e^{-\frac{f}{\gamma}}
    \dVol_g\label{deriv-classical-pseudo-gaussian}
    \end{align}
  is nonpositive, for $\gamma(t)\equiv 1$, $\mu(t)\equiv 1$, and some real $\alpha$.

Indeed, taking $\alpha$ to be 0, we have,
\begin{align*}
  \frac{1}{2}(1-\alpha)^2
  +\gamma^{-2}\left(\frac{1}{2}-\gamma+\dot\gamma\right)
&=  \frac{1}{2}(1-0)^2
+1^{-2}\left(\frac{1}{2}-1+0\right)\\
&=0,\\
\alpha+\frac{\dot\mu}{\mu}&=0+\frac{0}{1}\\
&=0.
\end{align*}
So \cref{deriv-classical-pseudo-gaussian} is zero and hence nonpositive.

  \end{proof}

Now, we consider how to exploit the ansatz \cref{complicated-derivative} optimally,
to give a pseudo-Gaussian bound for a factor $\gamma(t)$ whose growth rate is maximal
(for initial condition $\gamma(0)=1$).
It is apparent from the form of \cref{complicated-derivative}
(and the nonnegativity of the normalized potential $f$)
that we require
\[
\frac{1}{2}(1-\alpha)^2
+\gamma^{-2}\left(\frac{1}{2}-\gamma+\dot\gamma\right)
\le 0,
\]
and this constraint is weakest when $\alpha$ is taken to be 1
(so that $\frac{1}{2}(1-\alpha)^2$ is minimal).
In this case, we can maximize $\dot\gamma(t)$ by choosing $\gamma(t)$ to solve the ODE
\[
\frac{1}{2}-\gamma+\dot\gamma= 0;
\]
the solution to this ODE with initial condition $\gamma(0)=1$ is $\gamma(t):=\frac{e^t+1}{2}$.

Having reverse-engineered this choice of parameters,
here is the direct argument.

\begin{proposition}\label{optimal-pseudo-gaussian}
  Let $v\in C^2(M)$ be a function satisfying the bounds \cref{uniform-initial-bounds}.
  Then for all $t\geq 0$,
  \[
    \lVert P_tv\rVert_{L^2\left(e^{-\frac{f}{\frac{1}{2}(e^t+1)}}\dVol_g\right)}
    \le
    e^{nt/4}\lVert v\rVert_{L^2(e^{-f}\dVol_g)}.
  \]
\end{proposition}

\begin{proof}
  We will show that
  \[
    \mu(t)^{-n/2}
    \frac{d}{dt}
    \left(
      \mu(t)^{n/2}\int (P_tv)^2e^{-\frac{f}{\gamma(t)}}\dVol_g
    \right)
  \]
  is nonpositive, for
  \begin{align*}
    \gamma(t)&:= \frac{e^t+1}{2},\\
    \mu(t)&:=e^{-t},\\
    \alpha&:=1.
  \end{align*}

  By \cref{complicated-derivative}, it suffices to show that 
  \begin{align}
    \int_M
    \left\{
      \left[
    \frac{1}{2}(1-\alpha)^2
    +\gamma^{-2}\left(\frac{1}{2}-\gamma+\dot\gamma\right)
    \right]f
    +\left(\alpha+\frac{\dot\mu}{\mu}\right)\frac{n}{2}
  \right\}
    (P_tv)^2e^{-\frac{f}{\gamma}}
    \dVol_g\label{deriv-optimal-pseudo-gaussian}
    \end{align}
  is nonpositive.

Indeed, we have,
\begin{align*}
  \frac{1}{2}(1-\alpha)^2
  +\gamma^{-2}\left(\frac{1}{2}-\gamma+\dot\gamma\right)
  &=  \frac{1}{2}(1-1)^2
+\left( \frac{e^t+1}{2}\right)^{-2}\left(\frac{1}{2}- \frac{e^t+1}{2}+ \frac{e^t}{2}\right)\\
&=0+0\\
&=0,\\
\alpha+\frac{\dot\mu}{\mu}&=1+\frac{-e^{-t}}{e^{-t}}\\
&=0.
\end{align*}
So \cref{deriv-optimal-pseudo-gaussian} is zero and hence nonpositive.
  
    \end{proof}

Finally, we extend to general initial data by an approximation argument.

\begin{proof}[Proof of \cref{main-result}]
  The $C^2(M)$ functions satisfying the bounds \cref{uniform-initial-bounds} form a dense subspace of $L^2(e^{-f}M)$,
  and by \cref{optimal-pseudo-gaussian} the linear map $P_t$ is bounded (by $e^{nt/r}$) on this subspace.
  Therefore it has a bounded (by $e^{nt/r}$) extension to $L^2(e^{-f}\dVol_g)$;
  call this extension $Q$. 

  We argue that for any $v\in L^2(e^{-f}\dVol_g)$,
  it holds pointwise almost everywhere that $Qv=P_tv$,
  so that we can elide the distinction and simply speak of
  \[
    P_t:  L^2(e^{-f}\dVol_g)
    \to L^2(e^{-f}\dVol_g) \cap 
    L^2\left(e^{-\frac{f}{\frac{1}{2}(e^t+1)}}\dVol_g\right).
  \]

  Indeed, $v$ is approximated in $L^2(e^{-f}\dVol_g)$ by some sequence
  $(v_i)$ of $C^2(M)$ functions satisfying the bounds \cref{uniform-initial-bounds}.
  So $\left(P_t(v_i)\right)=\left(Q(v_i)\right)$ converges
  in $L^2(e^{-f}\dVol_g)$ to $P_t(v)$
  and in $L^2\left(e^{-\frac{f}{\frac{1}{2}(e^t+1)}}\dVol_g\right)$ to $Q(v)$,
  hence pointwise almost everywhere to both.
\end{proof}

\section{The heat equation along the Ricci flow}

\subsection{Transferring the estimate to a bound along the Ricci flow}\label{sec:transfer-bound}

In this section we prove \cref{nonincreasing-along-flow},
the statement of how estimate transfers to the Ricci flow starting from the soliton.

\begin{proof}[Proof of \cref{nonincreasing-along-flow}]
  Note that the relationship in the opposite direction to \cref{ricci-heat-from-drift-heat} is 
\[
  v(t,y)=u(-e^{-t},\psi_{-t}(y)).
\]

We have,
\begin{align*}
e^{-nt/2}\int v_t{}^2e^{-\frac{f}{(e^t+1)/2}}\dVol_g
&=
e^{-nt/2}\int u_{-e^{-t}}(\psi_{-t}(x)){}^2e^{-\frac{f}{(e^t+1)/2}}\dVol_g\\
&=
\int u_{-e^{-t}}(y){}^2e^{-\frac{f\circ \psi_t}{(e^t+1)/2}}\dVol_{e^{-t}\psi_t{}^*g}\\
&=
\int u_{-e^{-t}}(y){}^2e^{-\frac{e^{-t}f\circ \psi_t}{(1+e^{-t})/2}}\dVol_{e^{-t}\psi_t{}^*g}.
\end{align*}
Thus the statement follows from \cref{main-result}.
\end{proof}

\subsection{Example: Euclidean space}

In this section we illustrate \cref{nonincreasing-along-flow}, the application to the heat equation along the Ricci flow of a soliton,
by giving an elementary proof in the case of $\mathbb{R}^n$ with the classical heat equation.
We exploit a well-known criterion for the boundedness of an operator defined by a kernel:

\begin{theorem}[Schur's test]\label{kernels}
Let $(X,\mu)$, $(Y,\nu)$ be locally compact Hausdorff spaces equipped with positive Borel measures, such that $L^2(X,\mu)$, $L^2(Y,\nu)$ are separable Hilbert spaces, and $K:X\times Y\to\mathbb{R}$ a measurable function.
Suppose that there exists constants $C_X$, $C_Y$, and nonnegative measurable functions $h_X:X\to\Rr$, $h_Y:Y\to\Rr$, such that
\begin{eqnarray*}
\int_Y|K(x,y)|h_Y(y)d\nu_y&\leq C_Xh_X(x),&\quad \text{for almost all }x;\\
\int_X|K(x,y)|h_X(x)d\mu_x&\leq C_Yh_Y(y),&\quad \text{for almost all }y.
\end{eqnarray*}
Define an operator $T: L^2(Y,\nu)\to L^2(X,\mu)$
formally by,
\[
(Tf)(x)=\int_YK(x,y)f(y)d\nu_y;
\]
then the operator $T$ is bounded.
\end{theorem}

\begin{proposition} \label{heat-operator-euclidean}
  For each function $u_0\in L^2(\Rr^n, e^{-|\x|^2/4}d\x)$, there exists a solution $u:[-1,1)\times\Rr^n$ to the heat equation with initial condition $u(0,\cdot)=u_0$ and with $u(\tau,\cdot)\in L^2\left(\Rr^n, e^{-\frac{|\x|^2}{2\left(1-\tau\right)}}d\x\right)$.
  
  Moreover the heat evolution operator $u_0\mapsto u(\tau,\cdot)$ is bounded as a linear operator from $L^2(\Rr^n, e^{-|\x|^2/4}d\x)$ to $L^2\left(\Rr^n, e^{-\frac{|\x|^2}{2\left(1-\tau\right)}}d\x\right)$.
\end{proposition}

\begin{proof}
The kernel of the heat evolution operator, as a map between these function spaces, is
  \[
    K(\x,\y)=[4\pi(\tau+1)]^{-n/2}\exp\left(\frac{|\y|^2}{4}-\frac{|\x-\y|^2}{4(\tau+1)}\right);
  \]
that is, it differs from the standard heat kernel by a factor of $\exp\left(|\y|^2/4\right)$.

We set up for the use of Schur's test (Theorem \ref{kernels}) by writing
\begin{eqnarray*}
h_Y(\y)=\exp(|\y|^2/8),&&
h_X(\x)=\exp(|\x|^2/4\left(1-\tau\right)).\\
d\nu_\y=\exp\left(-|\y|^2/4\right)d\y,&&
d\mu_\x=\exp\left(-|\x|^2/2\left(1-\tau\right)\right)d\x.
\end{eqnarray*}

The algebraic identity
\[
\frac{2}{\tau+1}|\x-\y|^2-|\y|^2
=\frac{1-\tau}{\tau+1}\left\lvert \y-\tfrac{2}{1-\tau}\x\right\rvert^2 -\frac{2}{1-\tau}|\x|^2
\]
is easily checked. Using it twice,
\begin{eqnarray*}
K(\x,\y)h_Y(\y)d\nu_\y
&=&[4\pi (\tau+1)]^{-n/2}\exp\left(
-\frac{|\x-\y|^2}{4(\tau+1)}+\frac{|\y|^2}{8}
\right)
d\y\\
&=&[4\pi (\tau+1)]^{-n/2}\exp\left(
-\frac{1-\tau}{8(\tau+1)}\left\lvert \y-\frac{2}{1-\tau}\x\right\rvert^2 +\frac{1}{4(1-\tau)}|\x|^2
\right)d\y\\
&=&C\ h_X(\x)\exp\left(-
\frac{1-\tau}{8(\tau+1)}\left\lvert \y-\frac{2}{1-\tau}\x\right\rvert^2
\right)d\y.\\
% *********************
K(\x,\y)h_X(\x)d\mu_\x
&=&\ [4\pi(\tau+1)]^{-n/2}
\exp\left(
\frac{|\y|^2}{4}-\frac{|\x-\y|^2}{4(\tau+1)}
-\frac{|\x|^2}{4\left(1-\tau\right)}
\right)d\x\\
&=&\ [4\pi(\tau+1)]^{-n/2}
\exp\left(
\frac{|\y|^2}{8}
-\frac{1-\tau}{8(\tau+1)}\left\lvert \y-\frac{2}{1-\tau}\x\right\rvert^2
\right)d\x\\
&=&\ C\ h_Y(\y)\exp\left(
-\frac{1-\tau}{8(\tau+1)}\left\lvert \y-\frac{2}{1-\tau}\x\right\rvert^2
\right)d\x
\end{eqnarray*}
Integrating,
\begin{eqnarray*}
\int_Y|K(x,y)|h_Y(y)d\nu_y&\leq C_Xh_X(x),&\quad\text{for all }\x\in\Rr^n\\
\int_X|K(x,y)|h_X(x)d\mu_x&\leq C_Yh_Y(y),&\quad\text{for all }\y\in\Rr^n.
\end{eqnarray*}
The result follows by Theorem \ref{kernels}.
\end{proof}

% *************** Bibliography ***************

\bibliographystyle{alpha}
{\small\bibliography{analysis}{}}

\end{document}